\documentclass[12pt, reqno]{amsart}
\usepackage{ amsmath,amsthm, amscd, amsfonts, amssymb, graphicx, color}
\usepackage[bookmarksnumbered, colorlinks, plainpages]{hyperref}
\newtheorem{thm}{Theorem}[section]
\newtheorem{cor}[thm]{Corollary}
\newtheorem{lem}[thm]{Lemma}

\newtheorem{exam}[thm]{Example}
\numberwithin{equation}{section}

\begin{document}

\title{Additive properties of g-Drazin inverse for linear operators}

\author{Huanyin Chen}
\author{Marjan Sheibani$^*$}
\address{
Department of Mathematics\\ Hangzhou Normal University\\ Hang -zhou, China}
\email{<huanyinchen@aliyun.com>}
\address{Women's University of Semnan (Farzanegan), Semnan, Iran}
\email{<sheibani@fgusem.ac.ir>}

 \thanks{$^*$Corresponding author}


\begin{abstract}
In this paper, we investigate additive properties of generalized Drazin inverse for linear operators in Banach spaces. Under new polynomial conditions on generalized Drazin invertible operators $a$ and $b$, we prove their sum has generalized Drazin inverse and give explicit representations of the generalized inverse $(a+b)^d$. We then apply our results to $2\times 2$ operator matrices and consider the applications to the perturbation of generalized Drazin inverse. These extend the main results of Dana and Yousefi (Int. J. Appl. Comput. Math., {\bf 4}(2018), page 9), Yang and Liu (J. Comput. Appl. Math., {\bf 235}(2011), 1412--1417) and Sun et al. (Filomat, {\bf 30}(2016), 3377--3388).\end{abstract}

\maketitle

\section{Introduction}

Let $X$ be an arbitrary complex Banach space and $\mathcal{A}$ denote the Banach algebra $\mathcal{L}(X)$ of all bounded operators on $X$.
An element $a$ in $\mathcal{A}$ has g-Drazin inverse, i.e., generalized Drazin inverse, provided that there exists $b\in R$ such that $$b=bab, ab=ba, a-a^2b\in \mathcal{A}^{qnil}.$$ Here, $\mathcal{A}^{qnil}=\{a\in \mathcal{A}~|~1+ax\in \mathcal{A})~\mbox{is invertible for
every}~x\in comm(a)\}$. As is well known, $a\in \mathcal{A}^{qnil}\Leftrightarrow
\lim\limits_{n\to\infty}\parallel a^n\parallel^{\frac{1}{n}}=0$. Such $b$, if exists, is unique, and is called the g-Drazin inverse of $a$, and denote it by $a^d$. We always use $\mathcal{A}^d$ to stands for the set of all g-Drazin invertible $a\in \mathcal{A}$. The g-Drazin inverse of operator matrix has various applications in singular differential equations, Markov chains and iterative methods (see~\cite{B, C, CL, DC, DW, DW2, LQ, MD, MZ}). The motivation of this paper is to explore new additive properties of g-Drazin inverse for linear operators in Banach spaces. Furthermore, we apply our results to establish various conditions for the g-Drazin inverses of a $2\times 2$ partitioned operator matrices. Applications to the perturbation of g-Drazin inverse are obtained as well.

In Section 2, we present new polynomial conditions on generalized Drazin invertible operators $a$ and $b$, and show that their sum has generalized Drazin inverse and give explicit representations of the generalized inverse $(a+b)^d$. These extend the main results of Dana and Yousefi~\cite[Theorem 4]{D}, Yang and Liu ~\cite[Theorem 2.1]{Y} and Sun et al. ~\cite[Theorem 3.1]{S}. They are also the main tool in our following development.

In Section 3, we consider the generalized Drazin inverse of a $2\times 2$ operator matrix \begin{equation} M=\left(
\begin{array}{cc}
A&B\\
C&D
\end{array}
\right)\end{equation}
 where $A\in \mathcal{L}(X), D\in \mathcal{L}(Y)$. Here, $M$ is a bounded linear operator on $X\oplus Y$. This problem is quite complicated. It was expensively studied by many authors. We apply our results to establish new conditions under which $M$ has g-Drazin inverse. Our results contain many known results, e.g.,~\cite{DC} and~\cite{MD}.

If $A\in \mathcal{A}$ has g-Drazin inverse $A^d$. The element $A^{\pi}:=I-AA^d\in\mathcal{A}$ is called the spectral idempotent of $A$. Finally, in Section 4,
As an application of our results, we present new conditions with the perturbation $A^{\pi}$ under which $M$ has generalized Drazin inverse. These also extend ~\cite[Theorem 8]{D} to the g-Drazin inverse of operator matrices.

\section{Additive results}

The purpose of this section is to establish new conditions under which the sum of two g-Drazin invertible operators has g-Drazin inverse.
We begin with

\begin{lem} Let $a,b\in \mathcal{A}$ and $ab=0$. If $a,b\in \mathcal{A}^{d}$, then $a+b\in \mathcal{A}^{d}$ and
$$(a+b)^d=(1-bb^d)\big(\sum\limits_{n=0}^{\infty}b^n(a^d)^n\big)a^d+b^d\big( \sum\limits_{n=0}^{\infty}(b^d)^na^n\big)(1-aa^d).$$.\end{lem}
\begin{proof} See~\cite[Theorem 2.3]{DW2}.\end{proof}

In ~\cite{D}, Dana and Yousefi considered the Drazin inverse of $P+Q$ under the conditions that $PQP=0, QPQ=0, P^2Q^2=0$ and $PQ^3=0$ for complex matrices $P$ and $Q$. We note that every complex matrix has Drazin inverse which coincides with its g-Drazin inverse. We now extend this result to g-Drazin inverse of
operator matrices as follows.

\begin{thm} Let $a,b\in \mathcal{A}^{d}$. If $aba=0, bab=0, a^2b^2=0$ and $ab^3=0$, then $a+b\in \mathcal{A}^{d}$ and $$(a+b)^d=(1, b)M^d\left(
\begin{array}{c}
a\\
1
\end{array}
\right), M^d=F^d+G(F^d)^2+G^2(F^d)^3+G^3(F^d)^4,$$ where $$\begin{array}{c}
F^d=(I-KK^d)\big[\sum\limits_{n=0}^{\infty}K^n(H^d)^n\big]H^d+K^d\big[\sum\limits_{n=0}^{\infty}(K^d)^nH^n\big](I-HH^d);\\
H^d=\left(
\begin{array}{cc}
(a^d)^2&0\\
(a^d)^3&0
\end{array}
\right), K^d=\left(
\begin{array}{cc}
0&0\\
(b^d)^3&(b^d)^2
\end{array}
\right), G^4=0.
\end{array}
$$\end{thm}
\begin{proof} Set $$M=\left(
\begin{array}{cc}
a^3+a^2b+ab^2&a^3b\\
a^2+ab+ba+b^2&a^2b+ab^2+b^3
\end{array}
\right).$$ Then $$\begin{array}{lll}
M&=&\left(
\begin{array}{cc}
a^2b+ab^2&a^3b\\
0&a^2b+ab^2
\end{array}
\right)+\left(
\begin{array}{cc}
a^3&0\\
a^2+ab+ba+b^2&b^3
\end{array}
\right)\\
&:=&G+F.
\end{array}$$
We see that $G^4=0$ and $GF=0$. Moreover, we have
$$\begin{array}{lll}
F&=&\left(
\begin{array}{cc}
a^3&0\\
a^2+ab+ba+b^2&b^3
\end{array}
\right)\\
&=&\left(
\begin{array}{cc}
a^3&0\\
a^2+ba&0
\end{array}
\right)+\left(
\begin{array}{cc}
0&0\\
b^2+ab&b^3
\end{array}
\right)\\
&:=&H+K.
\end{array}$$
One easily check that $$H=\left(
\begin{array}{cc}
a^3&0\\
a^2+ba&0
\end{array}
\right)=\left(
\begin{array}{c}
a^2\\
a+b
\end{array}
\right) (a,0).$$ Since $(a,0)\left(
\begin{array}{c}
a^2\\
a+b
\end{array}
\right) =a^3\in \mathcal{A}^{d}$, it follows by Cline's formula (see~\cite[Theorem 2.1]{L}), we see that
$$\begin{array}{lll}
H^d&=&\left(
\begin{array}{c}
a^2\\
a+b
\end{array}
\right)((a^3)^d)^2(a,0)=\left(
\begin{array}{c}
a^2\\
a+b
\end{array}
\right)(a^d)^6(a, 0)\\
&=&\left(
\begin{array}{cc}
(a^d)^3&0\\
(a^d)^4+b(a^d)^5&0
\end{array}
\right).
\end{array}$$ Likewise, We have $$K^d=\left(
\begin{array}{c}
0\\
b
\end{array}
\right)(b^d)^4(1,b)=\left(
\begin{array}{cc}
0&0\\
(b^d)^3&(b^d)^2
\end{array}
\right).
$$ Clearly, $HK=0$. In light of Lemma 2.1, $$\begin{array}{l}
F^d=\\
(I-KK^d)\big[\sum\limits_{n=0}^{\infty}K^n(H^d)^n\big]H^d+K^d\big[\sum\limits_{n=0}^{\infty}(K^d)^nH^n\big](I-HH^d)
\end{array}$$ As $G^d=0$, by Lemma 2.1 again, we have $$M^d=F^d+G(F^d)^2+G^2(F^d)^3+G^3(F^d)^4.$$
Clearly, $M=\big(\left(
\begin{array}{c}
a\\
1
\end{array}
\right)(1, b)\big)^3$. By using Cline's formula, $$(a+b)^d=\big((1,b)\left(
\begin{array}{c}
a\\
1
\end{array}
\right)\big)^d=(1,b)M^d\left(
\begin{array}{c}
a\\
1
\end{array}
\right).$$ as asserted.\end{proof}

\begin{cor} Let $a,b, ab\in \mathcal{A}^d$ have g-Drazin inverses. If $a^2b=0$ and $ab^2=0$, then $a+b\in \mathcal{A}^{d}$.
\end{cor}
\begin{proof} Since $ab\in \mathcal{A}^d$, we see that $ba\in \mathcal{A}^d$ by Cline's formula. As $a^2(ab)=0$, it follows by Lemma 2.1  that $p:=a^2+ab\in \mathcal{A}^d$.
Likewise, $q:=ba+b^2\in \mathcal{A}^d$. One easily checks that $$pqp=0, qpq=0, p^2q^2=0~\mbox{and} ~pq^3=0.$$ In light of Theorem 2.2, $(a+b)^2=p+q\in \mathcal{A}^d$. According to~\cite[Corollary 2.2]{M}, $a+b\in \mathcal{A}^d$, as asserted.\end{proof}

Let $a,b\in \mathcal{A}^{d}$. If $aba=0, bab=0, a^2b^2=0$ and $a^3b=0$, then $a+b\in \mathcal{A}^{d}$. This is a symmetrical result of Theorem 2.1, and can be proved by a similar route.

\begin{lem} Let $a,b\in \mathcal{A}$ have g-Drazin inverses. If $aba=0$ and $ab^2=0$, then $a+b\in \mathcal{A}^{d}$.\end{lem}
\begin{proof} Let $p=a^2+ab$ and $q=ba+b^2$. Since $(ab)^2=0$, we see that $ab\in \mathcal{A}^{d}$. By Cline's formula, $ba\in \mathcal{A}^{d}$.
Clearly, $(ab)a^2=(ab)b^2=0$, it follows by Lemma 2.1 that $p,q\in \mathcal{A}^{d}$. Furthermore, we check that
$$pq=(a^2+ab)(ba+b^2)=a^2ba+a^2b^2+ab^2(a+b)=0,$$ and then $(a+b)^2=p+q\in \mathcal{A}^{d}$ by Lemma 2.1. According to~~\cite[Corollary 2.2]{M}, $a+b\in \mathcal{A}^{d}$, as required.\end{proof}

In ~\cite{Y}, Sun et al. the Drazin inverse of $P+Q$ in the case of $PQ^2=0, P^2QP=0, (QP)^2=0$ for two square matrices over a skew field.
As is well known, every square matrix over skew fields has Drazin inverse. We are now ready to extend \cite[Theorem 3.1]{Y} to g-Drazin inverses of bounded linear operators and prove:

\begin{thm} Let $a,b\in \mathcal{A}^{d}$. If $ab^2=0, a^2ba=0$ and $(ba)^2=0$, then $a+b\in \mathcal{A}^{d}$ and $$(a+b)^d=(1, b)M^d\left(
\begin{array}{c}
a\\
1
\end{array}
\right), M^d=F^d+G(F^d)^2+G^2(F^d)^3+G^3(F^d)^4,$$ where $$\begin{array}{l}
F^d=\\
(I-KK^d)\big[\sum\limits_{n=0}^{\infty}K^n(H^d)^n\big]H^d+K^d\big[\sum\limits_{n=0}^{\infty}(K^d)^nH^n\big](I-HH^d);\\
H^d=\left(
\begin{array}{cc}
(a^d)^2&0\\
(a^d)^3&0
\end{array}
\right), K^d=\left(
\begin{array}{cc}
0&0\\
(b^d)^3&(b^d)^2
\end{array}
\right), G^4=0.
\end{array}
$$\end{thm}
\begin{proof} Set $$M=\left(
\begin{array}{cc}
a^3+a^2b+aba&a^3b+abab\\
a^2+ab+ba+b^2&a^2b+bab+b^3
\end{array}
\right).$$ Then $$\begin{array}{l}
M=\\
\left(
\begin{array}{cc}
a^2b+aba&a^3b+abab\\
0&a^2b+bab
\end{array}
\right)+\left(
\begin{array}{cc}
a^3&0\\
a^2+ab+ba+b^2&b^3
\end{array}
\right)\\
:=G+F.
\end{array}$$
We see that $G^4=0, FGF=0$ and $FG^2=0$. Moreover, we have
$$\begin{array}{lll}
F&=&\left(
\begin{array}{cc}
a^3&0\\
a^2+ba&0
\end{array}
\right)+\left(
\begin{array}{cc}
0&0\\
b^2+ab&b^3
\end{array}
\right)\\
&:=&H+K.
\end{array}$$
As in the proof of Theorem 2.2, One easily checks that $$H^d=\left(
\begin{array}{cc}
(a^d)^3&0\\
(a^d)^4+b(a^d)^5&0
\end{array}
\right), K^d=\left(
\begin{array}{cc}
0&0\\
(b^d)^3&(b^d)^2
\end{array}
\right).$$ Moreover, $$\begin{array}{l}
F^d=\\
(I-KK^d)\big[\sum\limits_{n=0}^{\infty}K^n(H^d)^n\big]H^d+K^d\big[\sum\limits_{n=0}^{\infty}(K^d)^nH^n\big](I-HH^d)
\end{array}$$ In light of Lemma 2.1, $$M^d=F^d+G(F^d)^2+G^2(F^d)^3+G^3(F^d)^4.$$
Obviously, $M=\big(\left(
\begin{array}{c}
a\\
1
\end{array}
\right)(1, b)\big)^3$. By virtue of Cline's formula, $$(a+b)^d=\big((1,b)\left(
\begin{array}{c}
a\\
1
\end{array}
\right)\big)^d=(1,b)M^d\left(
\begin{array}{c}
a\\
1
\end{array}
\right),$$ as desired.\end{proof}

Let $a,b\in \mathcal{A}^{d}$. If $a^2b=0, aba^2=0$ and $(ba)^2=0$, then $a+b\in \mathcal{A}^{d}$. This can be proved in a symmetric way as in Theorem 2.5.

\section{Operator matrices}

To illustrate the preceding results, we are concerned with the generalized Drazin inverse for a operator matrix. Throughout this section, the operator matrix
$M$ is given by $(1.1)$, i.e., $$M=\left(
\begin{array}{cc}
A&B\\
C&D
\end{array}
\right),$$ where $A\in \mathcal{L}(X)^d, D\in \mathcal{L}(Y)^d$. Using different splitting approach, we shall obtain various conditions
for the g-Drazin inverse of $M$. In fact, the explicit g-Drazin inverse of $M$ could be computed by the formula in Theorem 2.5.

\begin{thm} If $ABC=0, DCA=0, DCB=0, CBCA=0$ and $CBCB=0$, then $M$ has g-Drazin inverse.\end{thm}
\begin{proof} Write $M=p+q$, where $$p=\left(
\begin{array}{cc}
A&B\\
0&D
\end{array}
\right),q=\left(
\begin{array}{cc}
0&0\\
C&0
\end{array}
\right).$$ It is obvious by ~\cite[Lemma 2.2]{DW2} that $p$ and $q$ have g-Drazin inverses. Clearly, $q^2=0$, and so $pq^2=0$.
As $ABC=0, DCA=0$ and $DCB=0$, then $p^2qp=0$. It follows from $CBCA=0$ and $CBCB=0$ that $(qp)^2=0$. Then by applying Theorem 2.5, $p+q=M$ has g-Drazin inverse.\end{proof}

\begin{cor} \cite[Theorem 3]{DC} If $BC=0$ and $DC=0$, then $M$ has g-Drazin inverse.
\end{cor}
\begin{proof} It is obvious by Theorem 3.1.\end{proof}

\begin{thm} If $ABC=0, ABD=0, DCB=0, BCBC=0$ and $BCBD=0$, then $M$ has g-Drazin inverse.\end{thm}
\begin{proof} Write $M=p+q$, where $$p=\left(
\begin{array}{cc}
A&0\\
C&D
\end{array}
\right),q=\left(
\begin{array}{cc}
0&B\\
0&0
\end{array}
\right).$$ By using ~\cite[Lemma 2.2]{DW2} it is clear that $p, q$ have g-Drazin inverses. Obviously,
$pq^2=0$. Also by the assumptions $ABC=0, ABD=0, DCB=0$ we have $p^2qp=0$. By using $BCBC=0$ and $BCBD=0$, we have $(qp)^2=0$. Then we get the result by Theorem 2.5.\end{proof}

\begin{cor} If $ABC=0, ABD=0, BCB=0$ and $DCB=0$, then $M$ has g-Drazin inverse.\end{cor}
\begin{proof} It is special case of Theorem 3.3.\end{proof}

If $AB=0$ and $CB=0$, we claim that $M$ has g-Drazin inverse (see ~\cite[Theorem 2]{DC}). This is a direct consequence of Corollary 3.4.

\begin{exam}
Let $M=\left(
\begin{array}{cc}
A&B\\
C&D
\end{array}
\right)$, where $$A=\left(
\begin{array}{ccc}
0&0&0\\
0&0&0\\
1&0&1
\end{array}
\right), B=\left(
\begin{array}{c}
1\\
1\\
-1
\end{array}
\right), C=\left(
\begin{array}{ccc}
1&0&1
\end{array}
\right)~\mbox{and}~D=0$$ be complex matrices. Then $ABC=0, ABD=0, BCB=0$ and $DCB=0$. In this case, $AB, CB\neq 0$.\end{exam}

\begin{lem} If $CBCB=0$, then $\left(
\begin{array}{cc}
0&B\\
C&0
\end{array}
\right)$ has g-Drazin inverse.\end{lem}
\begin{proof} Write $$\left(
\begin{array}{cc}
0&B\\
C&0
\end{array}
\right)=\left(
\begin{array}{cc}
0&0\\
C&0
\end{array}
\right)+\left(
\begin{array}{cc}
0&B\\
0&0
\end{array}
\right).$$ Let $p=\left(
\begin{array}{cc}
0&0\\
C&0
\end{array}
\right)$ and $q=\left(
\begin{array}{cc}
0&B\\
0&0
\end{array}
\right).$ In view of ~\cite[Lemma 2.2]{DW2}, $p$ has g-Drazin inverse. By virtue of Lemma 3.6, $q$ has g-Drazin inverse.
It is obvious that $pq^2=0$, $p^2qp=0$ and $(qp)^2=0$. Then by Theorem 2.5, $M$ has g-Drazin inverse.\end{proof}

\begin{lem} If $ABC=0$ and $CBCB=0$, then $\left(
\begin{array}{cc}
A&B\\
C&0
\end{array}
\right)$ has g-Drazin inverse.\end{lem}
\begin{proof} Write $$\left(
\begin{array}{cc}
A&B\\
C&0
\end{array}
\right)=\left(
\begin{array}{cc}
A&0\\
0&0
\end{array}
\right)+\left(
\begin{array}{cc}
0&B\\
C&0
\end{array}
\right).$$ Let $p=\left(
\begin{array}{cc}
A&0\\
0&0
\end{array}
\right)$ and $q=\left(
\begin{array}{cc}
0&B\\
C&0
\end{array}
\right)$. It is obvious that $pq^2=0$, $p^2qp=0$ and $(qp)^2=0$. Then by Theorem 2.5, it has g-Drazin inverse.\end{proof}

\begin{thm} If $ABC=0, DCA=0, DCB=0$ and $CBCB=0$, then $M$
 has g-Drazin inverse.\end{thm}
\begin{proof} Write $$M=\left(
\begin{array}{cc}
0&0\\
0&D
\end{array}
\right)+\left(
\begin{array}{cc}
A&B\\
C&0
\end{array}
\right).$$  Let $p=\left(
\begin{array}{cc}
0&0\\
0&D
\end{array}
\right)$ and $q=\left(
\begin{array}{cc}
A&B\\
C&0
\end{array}
\right)$. Then $p$ has g-Drazin inverse as $p^2=0$. In light of Lemma 3.7, $q$ has g-Drazin inverse. Also $pq^2=0$, $p^2qp=0$ and $(qp)^2=0$. Then by Theorem 2.5, $M$ has g-Drazin inverse. \end{proof}

\begin{cor} If $ABC=0, CBC=0, DCA=0$ and $DCB=0$, then $M$ has g-Drazin inverse.
\end{cor}
\begin{proof} it is clear by Theorem 3.8\end{proof}

\begin{lem} If $DCB=0$ and $CBCB=0$, then $\left(
\begin{array}{cc}
0&B\\
C&D
\end{array}
\right)$ has g-Drazin inverse.\end{lem}
\begin{proof} Write $$\left(
\begin{array}{cc}
0&B\\
C&D
\end{array}
\right)=p+q$$ where $p=\left(
\begin{array}{cc}
0&0\\
0&D
\end{array}
\right)$ and $q=\left(
\begin{array}{cc}
0&B\\
C&0
\end{array}
\right).$ In view of ~\cite[Lemma 2.2]{DW2}, $p$ has g-Drazin inverse. According to Lemma 3.6,
$q$ has g-Drazin inverse. Also $pq^2=0$, $p^2qp=0$ and $(qp)^2=0$. Then by Theorem 2.5, it has g-Drazin inverse.  \end{proof}

\begin{thm} If $ABC=0, ABD=0, DCB=0$ and $CBCB=0$, then $M$ has g-Drazin inverse.\end{thm}
\begin{proof} Write $$M=\left(
\begin{array}{cc}
A&0\\
0&0
\end{array}
\right)+\left(
\begin{array}{cc}
0&B\\
C&D
\end{array}
\right).$$ Clearly, $p$ has g-Drazin inverse. By Lemma 3.10, $q$ has g-Drazin inverse. From $ABC=0$ and $ABD=0$ we have $pq^2=0$, $p^2qp=0$ and $(qp)^2=0$. Therefore we complete the proof by Theorem 2.5.\end{proof}

As an immediate consequence, we derive

\begin{cor} If $ABC=0, ABD=0, BCB=0$ and $DCB=0$, then $M$ has g-Drazin inverse.
\end{cor}

\section{perturbation}

Let $M$ be an operator matrix $M$ given by $(1.1)$. It is of interest to consider the g-Drazin inverse of $M$ under generalized Schur condition
$D=CA^dB$ (see~\cite{S}). We now investigate various perturbation conditions with spectral idempotents under which $M$ has g-Drazin inverse.
We now extend~\cite[Theorem 8]{D} to the g-Drazin inverse of operator matrices.

\begin{thm} Let $A\in \mathcal{L}(X)^d, D\in \mathcal{L}(Y)^d$ and $M$ be given by $(1.1)$. If $CA^{\pi}AB=0, A^{\pi}A^2BC=0, A^{\pi}BCA^2=0, A^{\pi}BCB=0, ABCA^d=BCAA^d$ and $D=CA^dB$, then $M\in \mathcal{L}(X\oplus Y)^d$.\end{thm}
\begin{proof} Clearly, we have $$M=\left(
\begin{array}{cc}
A&B\\
C&CA^dB
\end{array}
\right)=P+Q,$$ where $$P=\left(
\begin{array}{cc}
AA^{\pi}&0\\
0&0
\end{array}
\right),Q=\left(
\begin{array}{cc}
A^2A^d&B\\
C&CA^dB
\end{array}
\right).$$ By assumption, we verify that $PQP=0, QPQ=0, P^2Q^2=0$ and $PQ^3=0$. Since $AA^{\pi}\in \mathcal{L}(X)^{qnil}$, we easily see that $P$ is quasinilpotent, and then it has g-Drazin inverse. Furthermore, we have
$$Q=Q_1+Q_2,~Q_1=\left(
\begin{array}{cc}
A^2A^d&AA^dB\\
CAA^d&CA^dB
\end{array}
\right),~Q_2=\left(
\begin{array}{cc}
0&A^{\pi}B\\
CA^{\pi}&0
\end{array}
\right)$$ and $Q_2Q_1=0$. Since $A^{\pi}BCA^2=0, A^{\pi}BCB=0$, we see that $(A^{\pi}BCA^{\pi})^2=A^{\pi}BCBCA^{\pi}-A^{\pi}BCA^2(A^d)^2=0$ and $(CA^{\pi}B)^2=CA^{\pi}BC(I-AA^d)B=CA^{\pi}BCB-CA^{\pi}BCA^2(A^d)^2B=0$. Therefore $Q_2^4=0.$ Moreover, we have $$Q_1= \left(
\begin{array}{c}
AA^d\\
CA^d
\end{array}
\right)\left(
\begin{array}{cc}
A&AA^dB
\end{array}
\right).$$ By hypothesis, we see that $$\left(
\begin{array}{cc}
A&AA^dB
\end{array}
\right)\left(
\begin{array}{c}
AA^d\\
CA^d
\end{array}
\right)=A^2A^d+AA^dBCA^d.$$

Since $A^{\pi}BCA^2=0$, we see that $(I-AA^d)BCA^2=0$, and so $BCA^2=AA^dBCA^2$. This implies that $BCA^d=AA^dBCA^d$, and so $$A^2A^d+AA^dBCA^d=A^2A^d+BCA^d.$$
Since $D=CA^dB$ has g-Drazin inverse, by Cline's formula, $BCA^d$ has g-Drazin inverse.
In view of~\cite[Theorem 2.1]{DW2}, $A^2A^d=A(AA^d)$ has g-Drazin inverse.

Since $ABCA^d=BCAA^d$, we check that $$\begin{array}{lll}
(A^2A^d)(BCA^d)&=&A(AA^dBCA^d)\\
&=&ABCA^d\\
&=&BCAA^d\\
&=&(BCA^d)(A^2A^d).
\end{array}$$
By virtue of~\cite[Theorem 2.1]{DW2},
 $A^2A^d+BCA^d$ has g-Drazin inverse. By using Cline's formula again, $Q_1$ has g-Drazin inverse. Therefore $Q$ has g-Drazin inverse. According to Theorem 2.2, $M$ has g-Drazin inverse, as asserted.\end{proof}

 \begin{cor} Let $A\in \mathcal{L}(X)^d, D\in \mathcal{L}(Y)^d$ and $M$ be given by $(1.1)$. If $CA^{\pi}AB=0, A^{\pi}A^2BC=0, A^{\pi}BCA^2=0, A^{\pi}BCB=0, A^2BCA=ABCA^2$ and $D=CA^dB$, then $M\in \mathcal{L}(X\oplus Y)^d$.\end{cor}
\begin{proof} As in the proof of Theorem 4.1, $BCA^d=AA^dBCA^d$. Since $A^2BCA=ABCA^2$, we have $$\begin{array}{lll}
ABCA^d&=&A(AA^dBCA^d)\\
&=&A^d(A^2BCA)(A^d)^2\\
&=&A^d(ABCA^d)(A^2A^d)\\
&=&BCA^d(A^2A^d)\\
&=&BCAA^d.
\end{array}$$
Therefore we complete the proof by Theorem 4.1.\end{proof}

Regarding a complex matrix as the operator matrix on ${\Bbb C}\times \cdots \times {\Bbb C}$, we now present a numerical example to demonstrate Theorem 4.1.

\begin{exam} Let $$A=\left(
\begin{array}{cccc}
1&0&0&0\\
0&0&0&0\\
0&0&0&0\\
0&1&1&0
\end{array}
\right), B=\left(
\begin{array}{cc}
1&0\\
1&-1\\
-1&1\\
1&-1
\end{array}
\right),$$
$$ C=\left(
\begin{array}{cccc}
1&1&1&1\\
1&-1&-1&1
\end{array}
\right) D=\left(
\begin{array}{cc}
1&0\\
1&0\\
\end{array}
\right)$$ be complex matrices and set
$$M=\left(
\begin{array}{cc}
A&B\\
C&D\\
\end{array}
\right).$$ Then
$$A^d=\left(
\begin{array}{cccc}
1&0&0&0\\
0&0&0&0\\
0&0&0&0\\
0&0&0&0
\end{array}
\right), A^{\pi}=\left(
\begin{array}{cccc}
0&0&0&0\\
0&1&0&0\\
0&0&1&0\\
0&0&0&1
\end{array}
\right).$$ We easily check that $$\begin{array}{c}
CA^{\pi}AB=0, A^{\pi}A^2BC=0, A^{\pi}BCA^2=0,\\
A^{\pi}BCB=0, ABCA^d=BCAA^d, D=CA^dB.
\end{array}$$ In this case, $A, D$ and $M$ have Drazin inverses, and so they have g-Drazin inverses.\end{exam}

By the other splitting approach, we derive

\begin{thm} Let $A\in \mathcal{L}(X)^d, D\in \mathcal{L}(Y)^d$ and $M$ be given by $(1.1)$. If $A^{\pi}A^2BC=0, A^{\pi}BCBC=0, A^{\pi}CABC=0, ABCA^d=BCAA^d$ and $D=CA^dB$, then $M\in \mathcal{L}(X\oplus Y)^d$.\end{thm}
\begin{proof} We easily see that $$M=\left(
\begin{array}{cc}
A&B\\
C&CA^dB
\end{array}
\right)=P+Q,$$ where $$P=\left(
\begin{array}{cc}
A&AA^dB\\
C&CA^dB
\end{array}
\right),Q=\left(
\begin{array}{cc}
0&A^{\pi}B\\
0&0
\end{array}
\right).$$ Then we check that $P^2QP=0, (QP)^2=0, Q^2=0$. Clearly, $Q$ has g-Drazin inverse. Moreover, we have
$$P=P_1+P_2,~P_1=\left(
\begin{array}{cc}
A^2A^d&AA^dB\\
CAA^d&CA^dB
\end{array}
\right),~P_2=\left(
\begin{array}{cc}
AA^{\pi}&0\\
CA^{\pi}&0
\end{array}
\right),$$ $P_2P_1=0$ and $P_2$ is quasinilpotent.
Since $A^d=A(A^d)^2$, we have $$P_1=\left(
\begin{array}{c}
AA^d\\
CA^d
\end{array}
\right)\left(
\begin{array}{cc}
A&AA^dB
\end{array}
\right).$$ By hypothesis, we see that $$\left(
\begin{array}{cc}
A&AA^dB
\end{array}
\right)\left(
\begin{array}{c}
AA^d\\
CA^d
\end{array}
\right)=A^2A^d+AA^dBCA^d.$$ As in the proof of Theorem 4.1, we easily check that $A^2A^d+AA^dBCA^d$ has g-Drazin inverse. Therefore $P_1$ has g-Drazin inverse.
By Lemma 2.1, $P$ has g-Drazin inverse. According to Theorem 2.5, $M$ has g-Drazin inverse.\end{proof}

\begin{cor} Let $A\in \mathcal{L}(X)^d, D\in \mathcal{L}(Y)^d$ and $M$ be given by $(1.1)$. If $A^{\pi}A^2BC=0, A^{\pi}BCBC=0, A^{\pi}CABC=0, A^2BCA=ABCA^2$ and $D=CA^dB$, then $M\in \mathcal{L}(X\oplus Y)^d$.\end{cor}
\begin{proof} As in the proof of Corollary 4.2, we prove that $ABCA^d=BCAA^d$. This completes the proof by Theorem 4.4.\end{proof}

\begin{cor} Let $A\in \mathcal{L}(X)^d, D\in \mathcal{L}(Y)^d$ and $M$ be given by $(1.1)$. If $A^{\pi}BC=0, A^2BCA=ABCA^2$ and $D=CA^dB$, then $M\in \mathcal{L}(X\oplus Y)^d$.\end{cor}
\begin{proof} This is obvious by Corollary 4.5.\end{proof}

\end{document}